\documentclass[letterpaper, 10 pt, conference]{ieeeconf}
\IEEEoverridecommandlockouts                              % This command is only needed if you want to use the \thanks command

\usepackage[dvips]{graphicx}

\usepackage{xcolor}
\usepackage{slashbox}
\usepackage{amsmath,amssymb,amsfonts,amsthm}
\usepackage{colortbl,longtable}
\usepackage{color}
\usepackage[normalem]{ulem}
\usepackage{soul}
\usepackage{setspace}
\usepackage{bm}
\usepackage{framed}
\usepackage{ifthen}
\usepackage[nice]{nicefrac}
\usepackage{hyperref}
\let\labelindent\relax
\usepackage{enumitem}
\usepackage{algorithm2e}

\newcommand{\hide}[1]{}
\newcommand{\old}[1]{}
\newcommand{\G}{{\cal G}}
\newcommand{\eop}{\hfill{$\blacksquare$}}

\newcommand{\muzero}{\mu_{0}}
\newcommand{\muone}{\mu_{1}}
\newcommand{\mutwo}{\mu_{2}}
\newcommand{\mua}{\mu_{a}}
\newcommand{\ce}{C_e}
\newcommand{\cp}{Q_p}
\newcommand{\cnp}{Q_{np}}
\newcommand{\cA}{{\cal A}}
\newcommand{\cD}{{\cal D}}
\newcommand{\bmu}{\bm{\mu}}
\newcommand{\td}{(\theta, \delta)\mbox{-success}}
\renewcommand{\S}{{\cal S}}

\newcommand{\na}{\mbox{non-adversarial}}

\newcommand{\fone}{\frac{\delta_a - \eta^*_{\widetilde{\theta}}\alpha_R}{\delta_a(1-\mua-\eta^*_{\widetilde{\theta}})}}
\newcommand{\xdoubleF}{\frac{1}{1-\alpha_F}\left(1-\mua-\frac{1}{cw\alpha_R \Delta^a}\right)}
\newcommand{\RAI}{{\mathcal R}_{AI}}

\newtheorem{thm}{Theorem}

\newtheorem{lemma}{Lemma}

\newtheorem{definition}{Definition}

\begin{document}
\title{Single-out fake posts: participation game and its design}
   \author{Khushboo Agarwal$^{1}$ and Veeraruna Kavitha$^{2}$, IEOR, IIT Bombay, India% <-this % stops a space
\thanks{$^{1}$Partially supported by Prime Minister’s Research Fellowship (PMRF), India.
        {\tt\small agarwal.khushboo@iitb.ac.in}}%
\thanks{$^{2}${\tt\small vkavitha@iitb.ac.in}}%
}

\maketitle
\thispagestyle{empty} 
%%%%%%%%%%%%%%%%%%%%%%%%%%%%%%%%%%%%%%%%%%%%%%%%%%%%%%%%%%%%%%%%%%%%%%%%%%%%%%%%%%%%%%
\begin{abstract}
% O are flooded with fake posts, which negatively impact society. 
Crowd-sourcing models, which leverage the collective opinions/signals of users on online social networks (OSNs), are well-accepted for fake post detection; however, motivating the users to provide the crowd signals is challenging, even more so in the presence of adversarial users. 

We design a participation (mean-field) game where users of the OSN are lured by a reward-based scheme to provide the binary (real/fake) signals such that the OSN achieves $(\theta, \delta)$-level of actuality identification (AI) - not more than $\delta$ fraction of non-adversarial users incorrectly judge the real post, and at least $\theta$ fraction of  non-adversarial users identify the fake post as fake. An appropriate warning mechanism is proposed to influence the decision-making of the users such that the resultant game has at least one Nash Equilibrium (NE) achieving AI. 
We also identify the conditions under which all NEs achieve AI.  Further, we numerically illustrate that one can always design an AI game if the normalized difference in the innate identification capacities of the users is at least $1\%$, when desired $\theta = 75\%$.

%  
% not many users mis-judge the fake post if the users can differentiate between fake and real posts to some extent.
\end{abstract}
%%%%%%%%%%%%%%%%%%%%%%%%%%%%%%%%%%%%%%%%%%%%%%%%%%%%%%%%%%%%%%%%%%%%%%%%%%%%%%%%%%%%%%
% \begin{IEEEkeywords}
% \textbf{Keywords:} Fake news detection, actuality identification, crowd-sourcing, participation

% \end{IEEEkeywords}
% \vspace{-3mm} 
%%%%%%%%%%%%%%%%%%%%%%%%%%%%%%%%%%%%%%%%%%%%%%%%%%%%%%%%%%%%%%%%%%%%%%%%%%%%%%%%%%%%%%
\section{Introduction}
It is well established that fake news has disastrous implications for society. The severity has increased multi-fold with the increasing usage of OSNs\footnote{$2.93$B people are active monthly on Facebook, out of $5.3$B people using the internet in 2022 (\cite{ITU, FB}).}, e.g., several fake news were either intentionally/unintentionally shared by the users  during the COVID-19 pandemic (\cite{cinelli2020covid}).   

A variety of machine/deep learning based algorithms have been proposed to control fake news propagation (see \cite{feng2022misreporting, sharma2019combating, ruchansky2017csi, ahmed2021detecting} and references therein); their focus is on text/content-based classification. However, it gets difficult to determine the actuality of intentional fake posts only based on content (\cite{sharma2019combating}). Further, most existing algorithms need training, while the datasets required for training the algorithms are limited, more so in specific languages (\cite{ahmed2021detecting}). Hence, algorithms whose applicability is not restricted due to limited datasets and language barriers are required.

Another main approach for fake news detection exploits the crowd signals or responses of the users (\cite{freire2021fake}). %Crowd signals or crowdsourcing refers to the collective application of human intelligence to solve complex issues . 
% One may think that such collective wisdom leads to polarization of ideas (due to social circles), but it is not valid for OSNs; users are often friends with friends of their friends and thus gain access to information from a distant part of the network \cite{landemore2012collective}. 
In general, crowd signals amplify the individuals' understanding and provide promising results (\cite{freire2021fake, kapsikar2020controlling}).  One can achieve much better results, if the crowd is further guided by a structured mechanism (\cite{kapsikar2020controlling}). To the best of our knowledge, only \cite{kapsikar2020controlling} attempts to guide the users in this manner, which strengthens the collective wisdom; we further build upon this idea here. %build on this idea by further considering the participation aspects of the crowd.

Authors in \cite{freire2021fake} mention that the limited users' willingness to publicly give their opinion is a limitation for the application of crowdsourcing models. %; an appropriate mechanism to get users' opinions is another limitation. 
This calls for the design of an appropriate participation game which sufficiently motivates the users to provide their responses.

% Authors in \cite{li2022dynamic} leverage on the content of the posts and the comments received on it to detect the fake posts. We propose an approach where simpler binary crowd signals (fake/real) are used to detect the veracity of the post. 

Recent algorithms also learn the credibility of the users based on the posts shared by them, while utilising their signals (\cite{freire2021fake, tschiatschek2018fake}). However, it is computationally expensive to learn each user's credibility on enormously large platforms like Facebook; hence, improved algorithms with less knowledge are required. Our mechanism requires just the knowledge of the fraction  of the adversarial users (who purposely mis-tag fake posts as real).% Furthermore, our algorithm ensures 

We, along with other authors in \cite{kapsikar2020controlling}, conducted an initial study towards singling-out fake posts using crowd-signals. A mechanism is designed where each user tags the post as real/fake based on its intrinsic ability to identify the actuality, the sender's tag and the system-generated warning. The warnings are generated by compiling the tags of the previous recipients. Our mechanism differs from that in \cite{kapsikar2020controlling} as: (i) the previous work assumed that the network has some prior knowledge about the actuality, while the OSN is oblivious in our case, (ii) we do not assume that all users participate in the tagging process, and (iii) no users foul played in \cite{kapsikar2020controlling}, while our analysis is robust against adversarial users.

Thus, motivated by \cite{kapsikar2020controlling}, we consider a mean-field  participation game among the users of the OSN. For each post on the user's timeline, the user is given a choice to tag the post as real/fake. The user can utilize the warning level of that post to make a more informed decision. The main objective is to design an actuality identification (AI) $(\theta, \delta)$ game where at some Nash Equilibrium (called AI-NE), at least $\theta$ fraction of  non-adversarial users tag the fake post as fake, and not more than $\delta$ fraction of non-adversarial users mis-tag the real post as fake. Towards this, the users are rewarded if $\td$ is achieved at NE and earn more if they consider warning level in their judgement.

We propose an easily implementable warning mechanism for the polynomial response function of the users. The designed AI game has at most two NE - one NE always exists and is always AI, and the other NE, if it exists, achieves the desired $\delta$-detection of the real posts. We also identify the conditions required to design an AI game.  %{\color{red}We numerically show that more than $70\%$ of non-adversarial users mis-tag the fake post at the non-AI-NE, when OSN demands that $80\%$ of the users correctly detect the fake news; this is true when the users are not very smart to differentiate between fake and real posts. }

% \vspace{-4mm}

%%%%%%%%%%%%%%%%%%%%%%%%%%%%%%%%%%%%%%%%%%%%%%%%%%%%%%%%%%%%%%%%%%%%%%%%%%%%%%%%%%%%%%
\section{System Description}\label{sec_system_desc}

Consider an online social network (OSN) where   content providers create and share  fake ($F$) or real ($R$) content in the form of posts. More often than not, neither the users of the OSN nor the network is aware that the given post is fake/real. Additionally, there is an \textit{adversary} in the system who designs fake posts and creates fake accounts/employs bots to confuse the other users about the actuality of its post by \textit{declaring fake posts as real}; \textit{they do not mis-tag the real post}. We refer all such fake accounts as adversarial ($a$) users. In presence of such $a$-users, the OSN is interested in designing a mechanism to detect the actuality of any given post. In particular, for any $\theta \gg \delta > 0$, the OSN aims to guide $\na$ users (referred to as just `users' or $na$-users) such that at least $\theta$-fraction of them correctly detect the fake post (denoted as $F$-post) as fake, and at maximum $\delta$-fraction of them consider the real post ($R$-post) as fake. 

Towards this, motivated by the work in \cite{kapsikar2020controlling}, we propose the following new warning mechanism. For each post shared on the OSN, the OSN additionally designs two tabs - the tag tab and the information tab (see Figure \autoref{fig:mechanism}). When a user clicks the former tab, it directly tags the post (as $R$ or $F$) based on its innate capacity to judge the actuality, while on clicking the latter tab, it tags additionally using the warning level. By virtue of the innate capacity, any user judges the $F$-post as fake with probability (w.p.) $\alpha_F > 0$ and mis-judges  the $R$-post as fake w.p. $\alpha_R > 0$; assume $\Delta_R > 1$ for $\Delta_u := \nicefrac{\alpha_F}{\alpha_u}$. Recall that $a$-users mis-tag any $F$-post w.p. $1$.

\textit{Observe that the difference between $\alpha_F$ and $\alpha_R$ might be minimal, due to which   users might not be able to accurately distinguish between the $F$/$R$-posts without any additional information.} The OSN aims to leverage upon this difference and accentuate the capacity of the users by providing them with additional information based on tags of other users (i.e., via collective wisdom) to single out the fake posts; this information is based on the responses of all the users, as the OSN can not differentiate between $a$-users and others. Now, the users can access this additional information by clicking on the information tab, after which a pop-up window appears with a warning level, $\omega$, whose design is discussed later in detail. The warning level influences the decision of the users to tag the post as real/fake. 

\vspace{-3mm}
\begin{figure}[htbp]
    \centering
    \includegraphics[trim = {1cm 1.7cm 1cm 1.7cm}, clip, scale=0.3]{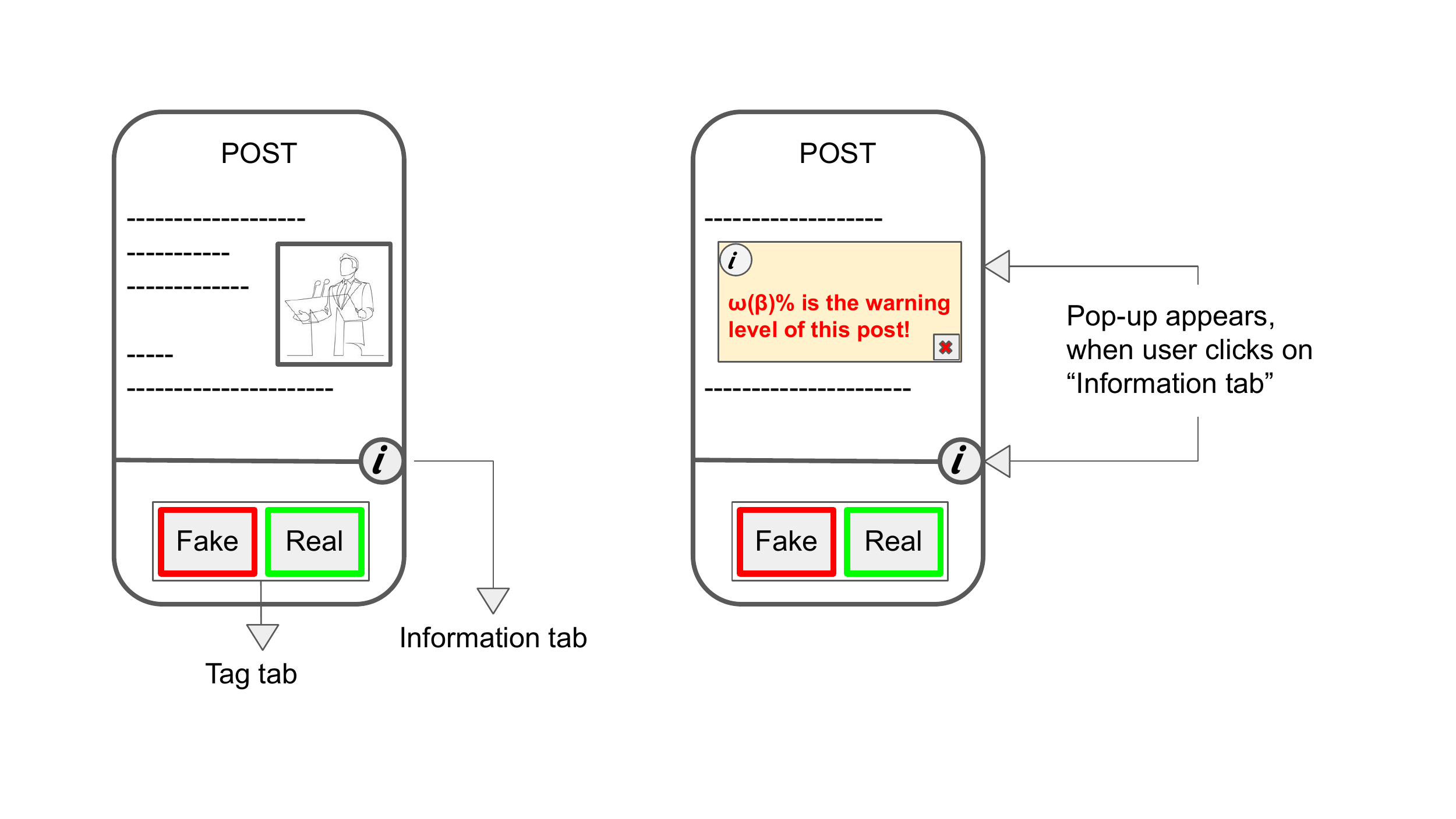}
    \caption{Warning mechanism for each post on OSN}
    \label{fig:mechanism}
    \vspace{-3mm}
\end{figure}
In reality, some users might not participate/tag, some might participate only based on their innate capacities, and the rest participate by considering both the warning provided by the information tab and their innate capacity. We represent the actuality of the underlying post as $u \in \{F, R\}$ and assume that the post is fake w.p. $p$, i.e., $P(u = F) = p$ with $p \in (0,1)$. Thus, the second type of users tag the $u$-post as fake w.p. $\alpha_u$. While, the third type of users tag the post as fake w.p. $r(\alpha_u, \omega(\beta))$ where $r: [0,1] \times \mathbb{R}^+ \mapsto [0, 1] $, the \underline{response function}, is defined as:
\begin{align}\label{eqn_response_func}
r(\alpha, \omega) = \min\{h(\alpha, \omega), 1\}.
\end{align}In the above, $h(\cdot, \cdot)$ is a Lipschitz continuous function and $\omega : [0,1] \mapsto \mathbb{R}^+ $ is the warning level designed by the OSN - the design of the warning depends on the fraction of fake tags so far.
The warning is continuously updated as more users tag the post; we discuss these dynamics in \autoref{sec_participation_game}. We anticipate that users perceive the post to be fake with higher probability, if the warning level is high. Further, the response ($r$) towards fake tagging the post should naturally increase with innate capacity, $\alpha_u$.

\subsection{Objectives of the OSN}\label{sec_objectives}
The OSN aims to design an appropriate warning mechanism. Its primary objective is to achieve \underline{$\td$}, for some $\theta > \alpha_F$ and $\delta > \alpha_R$, defined as follows:
\begin{itemize}
    \item if the underlying post is fake, then at least $\theta$ fraction of non-adversarial users tag it as fake, or
    \item if the underlying post is real, then at most $\delta$ fraction of non-adversarial users tag it as fake.
\end{itemize}Thus, the OSN is willing to compromise a slightly larger fraction ($\delta$) of mis-tags for the real posts, to obtain desired level ($\theta$) of identification of fake posts. However, one may have to ensure sufficient number of users participate in the tagging process. Towards this, the OSN induces a participation game among users, which we discuss next.

\section{Participation game by OSN}\label{sec_participation_game}
Let $n$ be the number of users on the OSN. Let $N_a$ be the $a$-users out of $n$ users. The remaining users on the OSN has three actions/strategies, $s$: (i) not to participate (say $s = 0$), (ii) participate and tag only based on intrinsic ability (say $s = 1$), and (iii) participate and tag on the basis of both intrinsic ability and the warning level (say $s = 2$); let us call the   users as type $0, 1, 2$ in the order of described actions. Let $N_i$ be the number of type $i$   users for $i \in \{0,1, 2\}$. Define $\mu_i := \nicefrac{N_i}{n}$ for $i \in \{0,1,2,a\}$; let $\mua \in [0,1)$ be fixed. \textit{We assume that OSN knows the fraction $\mua$ and nothing more. }%In fact, even the knowledge of an upper bound on $\mua$ is good.}

To motivate the   users to participate, the OSN provides publicly visible attributes to each user. For example, the OSN might reflect each user's average participation status on its profile, which gets updated with each post. We believe that such public recognition leads to pro-social behaviour (e.g., \cite{reiffers2019reputation}), i.e., participating in the tagging process. Thus, for the perceived stardom among its peers, each type $1, 2$ user receives a positive utility, say $\cp > 0$. 

The OSN further announces that the participants would get a total reward of $nR$, if the mechanism achieves $(\theta, \delta)$-success; furthermore, each type $2$ user gets $\gamma$ times the reward provided to the type $1$ user (where $\gamma > 1$); observe, $R$ and $ \gamma$ are  the design parameters for the OSN. \textit{We assume that $a$-users directly tag the post like type $1$ users; and as already mentioned, OSN can not differentiate between $a$ and type $1$ users.} Thus, the OSN provides the same reward to $a$-users, as it does to type $1$ users. Hence, upon success, a type $1$ user gets a reward of  $\nicefrac{nR}{(N_1 + N_a + \gamma N_2)} =  \nicefrac{R}{(\mu_1 + \mua +\gamma\mu_2)}$, and a type $2$ user gets $\nicefrac{\gamma R}{(\mu_1+ \mua + \gamma\mu_2)}$ reward. It usually takes long time for the OSN to ascertain the actuality of the posts, and thus, the reward is delivered to the users after the confirmation of the actuality.

The type $0$ users earn  a positive utility for the perceived comfort they experience by not participating and a negative utility for the public disapproval; this amounts to a consolidated utility $\cnp \in \mathbb{R}$ for each such user. We assume $\cp \geq \cnp$; if $\cp < \cnp$, then the OSN needs to provide an additional participation reward ($\geq \cnp - \cp > 0$) to all participants, irrespective of the outcome of the game. %We will see later that $\td$ is achieved only if $\cp \geq \cnp$. 

A type $2$ user also incurs a cost, $\ce > 0$, for the extra time invested in the process.%\footnote{One can consider $\ce = 0$ as well; however, we focus on the more interesting and realistic case in this paper.}

Thus, given the mechanism announced by the OSN, the   users make choices and participate in the tagging process. The tagging process may take some finite positive time, as the users tag asynchronously. But since the users make a \underline{participation choice} oblivious to the choices of the others, one can model it as a simultaneous move strategic form game, after capturing the tagging process by the fixed-point (FP) equation in \eqref{eqn_fp} described next. Observe here that the tagging process, and hence utility (of any user) depends on the strategy profile ${\bf s}=(s_1, \cdots, s_n)$ chosen by all the users.

Let $X_u$ be the number of fake tags provided by the users, for the $u$-post\footnote{Henceforth, we reserve sub-script $u$ for the actuality of the post, i.e., at places, we may not explicit write that $u \in \{R, F\}$.}, where $u \in \{R, F\}$. Define $\beta_u := \nicefrac{X_u}{(N_1+N_2 + N_a)}$ as the proportion of fake tags (observe, $N_1 + N_2 + N_a$ is the number of participants). The proportion satisfies the following FP equation: % depending on the actuality $u \in \{R, F\}$ of the post:
\begin{align}\label{eqn_fp}
    \beta_F(\bmu) &= \alpha_F \eta + (1-\eta-\eta_a) r(\alpha_F, \omega(\beta_F(\bmu))), \\
    \beta_R(\bmu) & = \alpha_R \eta + (1-\eta-\eta_a) r(\alpha_R, \omega(\beta_R(\bmu))), \mbox{ where} \nonumber \\
    \bmu &= (\mu_0, \mu_1, \mu_2), \ \eta = \eta(\bmu) := \frac{\muone}{\muone+\mutwo+\mua}, \mbox{ and } \nonumber \\
    \eta_a &= \eta_a(\bmu) := \frac{\mua}{\muone+\mutwo+\mua},\nonumber
\end{align}
with the terms explained as below:

$\bullet$ $\eta$ and $\eta_a$ represent the fraction of type $1$ and $a$-users among participants, and empirical measure $\bmu$ is defined by~${\bf s}$;

$\bullet$ Let $Z_{i}^{(j)}$ be the indicator that $i$-th user, among type $j$, has tagged the post as fake. The fraction of type $1$ users that tagged the post as fake equals $\sum_i \nicefrac{Z_i^{(1)}}{N_1} \approx \alpha_u$, when $n$ is sufficiently large  and $\mu_1 > 0$ (w.p.~$1$, by law of large numbers). Thus, the first term $\alpha_u \eta$ results from $\sum_i \nicefrac{Z_i^{(1)}}{(N_1+N_2+N_a)}$, for $u\in\{F, R\}$;

$\bullet$ While the tagging process is ongoing, the type $2$ users are provided refined warning based on the tags of the previous users. Thus, when $n$ is large, we anticipate the warning to stabilise. The stabilised value is reflected by the tags of the future users as well.  Then, $r(\alpha, \omega)$ is the response of the type $2$ users, at stabilised warning level $\omega$ which in turn results from the stabilised fraction of fake tags $\beta_u$. Thus, the second term (similar to the first term) approximately equals $\sum_i \nicefrac{Z_i^{(2)}}{(N_1+N_2+N_a)} \approx (1-\eta-\eta_a) r(\alpha, \omega(\alpha, \beta_u))$.  The overall fraction of fake tags ($\beta_u$) equals the sum of the corresponding terms; hence, the FP equation;

% $\bullet$ Since the $a$-users provide the fake tag to the $R$-post w.p. $1$, therefore, the last term $\eta_a$ results from $ \nicefrac{N_a}{(N_1+N_2+N_a)}$, when $n$ is large enough. And thus, 

$\bullet$ \textit{When the number of OSN users increases (i.e., as $n \to \infty$), the limit fraction of fake tags (for any given strategy profile of the users) is indeed given by the FP equation \eqref{eqn_fp}}; Lemma \ref{lemma_beta} of next subsection proves this.

The OSN is   interested in the fake tags from $na$-users only, $\nicefrac{X_F}{(N_1+N_2)}$, but, it can not distinguish between the tags from $na$ and $a$-users; it can observe only the overall fraction of fake tags, $\beta_u$. Hence, \textit{$\td$ is redefined as follows} in terms of $\beta_u$ (see \eqref{eqn_fp}):

\vspace{-4mm}
{\small
\begin{align}\label{eqn_td}
\begin{aligned}
    \beta_F(\bmu) &\geq \theta \left(1 - \eta_a(\bmu)\right) := \theta_a(\bmu), \mbox{ and} \\
    \beta_R(\bmu) &\leq \delta \left(1 - \eta_a(\bmu)\right) := \delta_a(\bmu).
\end{aligned}
\end{align}}%In the above, $\beta_u(\bmu)$ is the solution of  \eqref{eqn_fp}, when $n$ is large. 
Thus, $\td$ depends on $\textbf{s}$ only via $\bmu$, when $n$ is large (also, see footnote \ref{footnote_uni_dev}). Let $P_{\bmu}(\cdot) := P(\cdot|\bmu)$ be the corresponding conditional probability. From \eqref{eqn_td}, the probability of mechanism being $\td$ful, given the empirical ($\textbf{s}$-dependent) distribution $\bmu$ and parameterised by $(\theta, \delta)$, is given by (recall $p$ is the probability of underlying post being fake):

\vspace{-4mm}
{\small
\begin{eqnarray}\label{eqn_prob_success}
P_{\bmu}(S; \theta, \delta) = p  P_{\bmu}(\beta_{F, \overline{k}} \geq \theta_a) +  (1-p)  P_{\bmu}(\beta_{R, \overline{k}} \leq \delta_a), 
\end{eqnarray}}where $\beta_{u, k}$ for any  $k$ represents the proportion of the fake tags for the $u$-post, immediately after the  $k$-th participant tags (see details in sub-section \ref{subsec_ode}), and ${\overline k} = n(1-\mu_0)$ is the index of the last participant to tag. Recall that the game is played continually, however, OSN will design a mechanism  anticipating the eventual responses of the users.
Finally, the utility of $i$-th user is\footnote{\label{footnote_uni_dev}Here, the influence of a single user's action is negligible, as is usually the case in MFGs; observe that if, for example, $\muone, \mutwo$ fractions correspond to $(s_i = 1, s_{-i})$, then the fractions corresponding to $(s_i = 2, s_{-i})$ equal $\muone - \frac{1}{n}, \mutwo + \frac{1}{n}$ which respectively converge to $\muone, \mutwo$ as $n \to \infty$.} ($\textbf{s} = (s_i, \textbf{s}_{-i})$, a standard game theoretic notation):
\begin{align}\label{eqn_util}
U(s_i, \textbf{s}_{-i}) = 
\begin{cases}
\cnp, &\mbox{if } s_i = 0,\\
\cp + \frac{R P_{\bmu}(S; \theta, \delta)}{\muone + \mua + \gamma \mutwo}, &\mbox{if } s_i = 1,\\
\cp - \ce + \frac{\gamma R P_{\bmu}(S; \theta, \delta)}{\muone + \mua + \gamma \mutwo}, &\mbox{if } s_i = 2.
\end{cases}
\end{align}
% Observe that we do not model the utility of $a$-users, as they do not strategically (but rather determinedly) choose to be an adversary. 
This completes the description of the participation game represented by $\left<\{1, \dots,  n-N_a\}, \{0,1,2\}, (U_i) \right>$ for any given $N_a$. We derive the solution of this game, Nash Equilibrium\footnote{
\begin{definition}\cite{narahari2014game}
Given a strategic form game $\left<\{1, \dots, n\}, (S_i), (U_i) \right>$, the strategy profile $\textbf{s}^* = (s_i^*)_{i = 1}^n$ is called a \underline{pure
strategy Nash equilibrium} if $U_i(s^*) = \mbox{argmax}_{s\in S_i} U_i(s, \textbf{s}_{-i}^*)$ for each $i \in \{1, \dots, n\}$.
\end{definition}} (NE).
From \eqref{eqn_util}, $U(s_i, \textbf{s}_{-i}) = U(s_i, \bmu)$ for $n$ large enough (see footnote \ref{footnote_uni_dev}). As a result, the utility of any user depends on its strategy and the relative proportion of users, $\bmu$. It is thus appropriate to analyse such large population game using mean field game (MFG) theory. For MFGs with countable action set (as described below), the solution concept is again NE, and is equivalently given by the following (see \cite{carmona2018probabilistic}):
\begin{definition}\label{defn_NE_MFG}
Consider a mean field game, with a countable strategy set, $S$. Let $\mu_s$ represents the fraction of players choosing action $s$, for some $s\in S$. Further, let the utility of any player be $U(s, \bmu)$, for $\bmu = (\mu_s)_{s \in S}$. % being the empirical measure on the choices made by all the players.
Then, $\bmu^* = (\mu_s^*)_{s \in S}$ is called a \underline{Nash equilibrium of the MFG} if $\S(\bmu^*) \subseteq \mbox{argmax}_s U(s, \bmu^*)$, where $\S(\bmu) := \{s : \mu_s > 0\}$ represents the support. \label{defn_NE}
\end{definition}
We now consider the MFG variant of the participation game by letting $n \to \infty$. The utilities remain as in \eqref{eqn_util}, but the probability of success\footnote{The $\lim_{k \to \infty} P_{\bmu}(\cdot)$ may not exist for all distributions $\bmu$, and hence it is appropriate to define probability of success with $\liminf$ in \eqref{eqn_prob_success}.} changes to (when $\mu_0 < 1-\mua$):
\begin{align}\label{eqn_prob_success_n_large}
\begin{aligned}
P_{\bmu}(S; \theta, \delta) &= p \liminf_{k \to \infty} P_{\bmu}(\beta_{F, k} \geq \theta_a) \\
&\hspace{0.6cm}+  (1-p) \liminf_{k \to \infty} P_{\bmu}(\beta_{R, k} \leq \delta_a),
\end{aligned}
\end{align}
and $P_{(1-\mua, 0, 0)}(S; \theta, \delta) := 0$.
Let $\bmu^* = (\muzero^*, \muone^*, \mutwo^*)$ be the NE of MFG (when its exists), where $\mu_i^* \geq 0$ for $i \in \{0,1,2\}$ and $\sum_{i=0}^2 \mu_i^* = 1-\mua$. The OSN aims to appropriately design ($R, \gamma, \omega$) such that the equilibrium outcome of the resultant game achieves $\td$; \textit{we represent the game compactly by $\G(R, \gamma, \omega)$}. We call the MFG as Actuality identification (AI) game as per the definition below:
\begin{definition}\label{defn_AI}
Call a game $\G(R, \gamma, \omega)$ to be an \underline{AI game} if there exists parameters $R > 0, \gamma \geq 1$ and a warning mechanism $\omega$ such that for some NE $\bmu^*$ of the participation game  the following is true:

\vspace{-4mm}
{\small
\begin{align*}
\liminf_{k \to \infty} P_{\bmu^*}(\beta_{F, k} \geq \theta_a) &= 1 \mbox{ and }
\liminf_{k \to \infty} P_{\bmu^*}(\beta_{R, k} \leq \delta_a) = 1.
\end{align*}}
% where $\theta_a^* := \theta_a(\bmu^*)$ and $\delta_a^* := \delta_a(\bmu^*)$. 
Such an NE is called an \underline{AI-NE}.
\end{definition}
In simpler words, a participation game is called an AI game if at some NE, $\td$ is achieved. %at least $\theta$ fraction of $\na$ users identify the $F$-post as fake, and at most $\delta$ fraction of them identify the $R$-post as fake.
Observe that the definition \ref{defn_AI} requires that $P_{\bmu^*}(S;\theta, \delta) = 1$ (see \eqref{eqn_prob_success_n_large}), which implicitly demands that the random tagging-dynamics driven by the warning mechanism leads to $\td$ w.p. $1$. In other words, the two limit infimums in \eqref{eqn_prob_success} equal $1$.

We derive the sufficient conditions for designing an AI game, after providing the analysis of the MFG in the next section. 
For the sake of clarity, we re-state that design parameters of the system are $R, \gamma$, along with the warning mechanism $\omega$, the parameters $\alpha_R, \alpha_F, \cp, \cnp, \ce$ are user specific, while $p$ is a post specific parameter. We henceforth assume the following: 

\noindent \textbf{(P)} Assume $R > 0, \gamma > 1$, $\alpha_F \in (\alpha_R, 1)$, $\cnp \leq \cp$, $\ce > 0$, $p \in (0,1)$, $\mua \in [0,1)$. %Further, assume $\theta \in \left(\alpha_F, 1\right]$, $\delta \in (\alpha_R, \theta)$.

% \noindent Henceforth, we provide the analysis under assumption \textbf{(P)}.

\section{MFG: Analysis and Design}\label{sec_MFG}
We first derive the limit of $\beta_{u, k}(\bmu)$ that defines \eqref{eqn_prob_success_n_large} for any $\bmu$. Then, we appropriately choose $R, \gamma$ and $\omega$ such that the resultant is an AI game.

\subsection{Tagging dynamics}\label{subsec_ode}

For any given $\bmu$, recall that users asynchronously visit the OSN and provide the tag; some users also utilise the warning level.\footnote{We skip explicit mention of the dependence of various entities (e.g., $\eta, \eta_a$) on  $\bmu$ at few places for simplifying notations and improving on clarity.} This leads to continuous-time evolution of the proportion of fake tags, ($\beta_{u, k}$), and the corresponding warning levels, $(\omega(\beta_{u, k}))$. However, it is sufficient to observe the tagging process whenever a user decides to tag, i.e., the embedded process; let $k \in \mathbb{Z}^+$ be the index of such decision epochs. The time duration between two decision epochs must follow some distribution, however it's specific details are immaterial for the study. At any decision epoch, the participant can be an $a$-user, type $1$ or $2$ user w.p. $\eta_a$, $\eta$ or $1-\eta-\eta_a$ respectively.

Let $X_{u,k}$ be the number of fake tags for $u$-post at $k$-th epoch, where $u \in \{F, R\}$ is fixed. The fraction of fake tags at $(k+1)$-th epoch, $\beta_{u, k+1}(\bmu)$, can then be written as:
\begin{align}\label{eqn_dynamics_beta}
\begin{aligned}
    \beta_{u, k+1} &:= \frac{X_{u,k+1}}{k+1} = \frac{ X_{u,k} + 1_{\{\mbox{tag for $u$-post} = F\}}}{k+1} \\
    &= \beta_{u,k} + \frac{1}{k+1} L_{u,k}, \mbox{ where }  \\
    L_{u,k} &:= 1_{\{\mbox{tag for $u$-post} = F\}} - \beta_{u,k}.
\end{aligned}
\end{align}
This iterative process can be analysed using stochastic approximation tools (see \cite{kushner2003stochastic}). Ordinary Differential Equation (ODE) based analysis is a common approach to study such processes - define the conditional expectation of $L_{k, u}$ with respect to sigma algebra, ${\mathcal F}_{k, u} := \sigma\{X_{u,j}: j < k\}$:
\begin{align}\label{eqn_cond_exp}
\begin{aligned}
    E[L_{u,k}|{\mathcal F}_k] &= g_u(\beta_{u, k}), \mbox{ where}\\
g_u(\beta) := 
     \alpha_u \eta + (1-&\eta-\eta_a) r(\alpha_u, \omega(\beta)) - \beta. 
%      , \\
% g_R(\beta) := 
%      \alpha_R \eta + (1-&\eta-\eta_a) r(\alpha_R, \omega(\beta)) - \beta.
\end{aligned}
\end{align}
Then, the dynamics in \eqref{eqn_dynamics_beta} can be captured via the following autonomous ODE (proved in Lemma \ref{lemma_beta} given below):
\begin{align}\label{eqn_ode}
    \dot{\beta_u} = g_u(\beta_u).
\end{align}
The right hand side of the above ODE is Lipschitz continuous, and thus has unique global solution (see \cite[Theorem 1, sub-section 1.4]{piccini1984ordinary}).
Define the domain of attraction (DoA)

\vspace{-4mm}
$$\cD_u := \{\beta \in [0,1]: \beta_u(t) \stackrel{t \to \infty}{\longrightarrow} \cA_u, \mbox{ if } \beta_u(0) = \beta\},$$for asymptotically stable (AS) set $\cA_u$ of the ODE \eqref{eqn_ode} in the interval $[0, 1]$ (see \cite{piccini1984ordinary}).
Assume the following: 

\noindent \textbf{(A)} $P(\beta_{u, k} \in \cD_u \mbox{ infinitely often}) = 1.$

\begin{lemma}\label{lemma_beta}
Under assumption \textbf{(A)}, the sequence $(\beta_{u,k})$ converges to $\cA_u$ w.p. $1$, as $k \to \infty$.  \eop
\end{lemma}
% \noindent \textit{Proof:} See \TR{\cite[Appendix]{arxiv}}{proof in Appendix (section \ref{appendix})}. 
The proof of above Lemma and all further results are in Appendix (section \ref{appendix}).
It is clear that the attractor set ($\cA_u$) depends on the choice of warning mechanism $\omega(\cdot)$, $\bmu$ and the response function $r(\cdot, \cdot)$. \textit{Henceforth, we assume $\cA_u$ to be a singleton (for each $\bmu$) and provide the analysis; we prove this assumption and \textbf{(A)} for a special response function $r$ and the correspondingly chosen $\omega$ in section \ref{sec_response1}}. One needs to characterise $\cA_u$ and prove \textbf{(A)} for other response functions considered in future. Thus, by the virtue of Lemma \ref{lemma_beta}, $\beta_u \in \cA_u$ uniquely signifies the eventual fraction of fake tags for the $u$-post.

Hence, if $\cA_u = \{\beta_u^*\}$ for both $u$, then \eqref{eqn_prob_success_n_large} simplifies to:
\begin{align}\label{eqn_simplified_prob_success}
P_{\bmu}(S; \theta, \delta) = p 1_{\{\beta_{F}^* \geq \theta_a(\bmu)\}} + (1-p) 1_{\{\beta_R^* \leq \delta_a(\bmu)\}}.
\end{align}

\subsection{Design of AI game}\label{sec_analysis}
Any $\bmu$ that satisfies definition \ref{defn_NE_MFG} qualifies to be a NE of the MFG. However, the OSN is interested in designing a game (i.e., choosing $R, \gamma, \omega$) for which at least one AI-NE exists (see definition \ref{defn_AI}). We achieve the same in Theorem \ref{thrm_AI}. We will further derive the conditions under which the game has only AI-NE for a specific response function in section \ref{sec_response1}.

Before we state the result, define \underline{$\bmu_x := (0,x, 1-x-\mua)$}  and \underline{$\beta_u^x := \beta_u^*(\bmu_x)$}, i.e., the attractor obtained via Theorem \ref{thrm_AI} for $\bmu = \bmu_x$, for any $x \in [0, 1-\mua]$, for $u\in\{R,F\}$.

Consider a response function $r(\alpha, \omega)$. Suppose there exists a warning mechanism $\omega$   that satisfies the following conditions related to the ODE \eqref{eqn_ode}  for some $\eta \in (0,1-\mua)$:

\noindent (B.i) $\exists$ a  $\beta_R^\eta  \in [0,\delta_a(\bmu_\eta)]$ such that $g_R(\beta_R^\eta) = 0$, and 

\noindent (B.ii) $\exists$ a $\beta_F^\eta \in [\theta_a(\bmu_\eta),1]$ such that $g_F(\beta_F^\eta) = 0$,

\noindent (B.iii) there are no other equilibrium points of the ODE \eqref{eqn_ode} in $[0,1]$ for each $u$, with respect to $\bmu_\eta$, and

\noindent (B.iv) $\frac{\partial g_u(\beta_u^\eta)}{\partial \beta_u^\eta} < 0$ for each $u$. 

For such a pair of $(r, \omega)$, one can anticipate that the tagging dynamics converge to a unique limit point for each $u$, leading to an AI game. This indeed is true as claimed below (proof in  Appendix (section \ref{appendix})).

\begin{thm}\label{thrm_AI}
Consider a pair of response and warning mechanism that satisfy assumptions (B.i)-(B.iv). Then, $\beta_u^\eta$ is AS with DoA as $[0,1]$ for each $u$. Further, choose $R, \gamma$:
\begin{align}\label{eqn_R_gamma}
\begin{aligned}
\gamma &> \underline{\gamma}(\eta) := \frac{1}{1-p}\left(\frac{1-(\eta+\mua)(1-p)}{1-\eta-\mua} \right) \mbox{ and }\\
R &= \ce \left(1 -\eta - \mua + \frac{1}{\gamma-1}  \right).
% \gamma &> \underline{\gamma}(\eta) :=               \frac{1}{1-p}\left(1+\frac{p(\eta+\mua)}{1-(\eta+\mua)} \right).
\end{aligned}
\end{align}Then, $\G(R, \gamma, \omega)$ is an AI game with the following properties:

\noindent (a) $\bmu_\eta$ is an AI-NE,

\noindent (b) any $\bmu_x$, with $x \in [0, \eta)\cup\{1-\mua\}$ is not an NE, and 

\noindent (c) any $\bmu$ with $\mu_0 > 0$ can not be a NE.  \eop
\end{thm}

Thus, if the OSN chooses $(R, \gamma, \omega)$ as per Theorem \ref{thrm_AI} then the resultant is an AI game.  Observe that at AI-NE ($\bmu_\eta$), there is a non-zero proportion of both type $1$ and $2$ users. In fact, the OSN is able to motivate all the $na$-users to participate in the tagging process at $\bmu_\eta$. %No NE emerges with non-zero proportion of type $0$ users. 

% The OSN (monetarily) benefits if type $1$ users increase ($R$ decreases with $\eta$), however as seen from the Theorem \ref{thrm_response1}, $\bmu$ with $\mu_1 = 1-\mua$ is not an NE. 
It can be seen from \eqref{eqn_R_gamma} that the OSN can monetarily benefit (reduced $R$) by either choosing a larger $\gamma$ (a bigger disparity between rewards provided to type $1$, $2$ users) or larger $\eta$ (that leads to larger fraction of type $1$ users). Interestingly, $R$ can be reduced to an arbitrarily small value (observe infimum of the achievable $R$ equals $0$). Further even if the perceived cost of processing warning ($\ce$) is high, one can design the desired AI game by appropriately scaling $\gamma, \eta$ with the same reward, $R$. % the parameters.

Any NE requires a mixed behavior; if all the users consider the warning ($\mu_2 = 1-\mua$) or tag only based on their intrinsic abilities ($\mu_1 = 1-\mua$), then, there is no NE. 

Ideally, the OSN would want to design a game where any NE, $\bmu_x^*$, is an AI. Theorem \ref{thrm_AI} provides such guarantees only for $x \in [0, \eta]\cup\{1-\mua\}$ (any such $x \neq \eta$ is not a NE). We next delve upon the remaining configurations for a specific response function. 

% IMPORTANT -------------------

% The conditions (i), (ii) and (iv) in Theorem \ref{thrm_AI} are necessary as they imply that the equilibrium points are AS for the ODE \eqref{eqn_ode}. However, condition (iii) can be relaxed if there exist more than one stable point of the ODE \eqref{eqn_ode} which also achieve $\td$. 

% The challenge arises when there can exist a stable point (say $\overline{\beta}_u^\eta$) which does not achieve $\td$. One then needs to design an appropriate additional mechanism such that the ODE solution does not converge to $\overline{\beta}_u^\eta$; note that even then the tagging dynamics can converge to $\overline{\beta}_u^\eta$. We leave analysis of such cases for the future. Nevertheless, we prove that there exists a unique attractor for each $u$ for the special response function considered in the next section.

%%%%%%%%%%%%%%%%%%%%%%%%%%%%%%%%%%%%%%%%%%%%%%%%%%%%%%

\vspace{-4mm}
\section{A specific response function}\label{sec_response1}
In this section, we specifically consider a class of polynomial response functions, $r(\alpha, \omega) = \min\{h(\alpha, \omega), 1\}$ with $h(\cdot, \cdot)$ defined as (extension of the linear response in \cite{kapsikar2020controlling}):
\begin{align}\label{eqn_response1}
    h(\alpha, \omega) = c \alpha^a \omega^b, \mbox{ where } a, b, c \in \mathbb{R}^+.
\end{align}Recall that any type $1$ user fake tags a $u$-post w.p. $\alpha_u$. However, if a user incorporates warning level as well, then it responds differently - it fake tags the $u$-post w.p. $r(\alpha_u, \omega)$. We model the said effect in \eqref{eqn_response1} via $a, b$ which indicates the positive correlation of user's response to the innate capacity and the warning level respectively. Next, we introduce few notations:

\vspace{-4mm}
{\small\begin{align}\label{eqn_notations}
% \begin{aligned}
    \eta^*_l &:=  \frac{(1-l)(1-\mua)}{1-\alpha_F} \mbox{ for any } l \in \mathbb{R}, \nonumber\\
    K_\delta &= \kappa^2-4\delta\alpha_R\alpha_F(\Delta_R)^a \geq 0 \mbox{ for }  \nonumber\\
    \kappa &:= \delta\left((\Delta_R)^a(1-\alpha_F) - 1 \right) - \alpha_R (\Delta_R)^a, \\
    \delta_a &:= \delta(\bmu_x) = \delta(1-\mua) \mbox{ for any } x \in (0, 1-\mua), \mbox{ and}  \nonumber\\
    \overline{\eta} &:= \frac{\delta_a((1-\mua)cw\alpha_R - 1)}{cw\alpha_R \delta_a - \alpha_R}.  \nonumber
% \end{aligned}
\end{align}}

We choose the following warning mechanism that would modulate users' response given in \eqref{eqn_response1} for $\td$ful identification of the posts (see Theorem \ref{thrm_response1}):
\begin{align}\label{eqn_warning}
    \omega( \beta) = w^{1/b} \alpha_R^{(1-a)/b} \beta^{1/b},
\end{align}where $w$ will be appropriately chosen as per Algorithm \ref{alg_AI}. Note that while designing the warning mechanism ($\omega$), we assume that the OSN knows $\alpha_R$. 

\vspace{-0.4cm}
\RestyleAlgo{ruled}
\begin{algorithm}
\caption{Design of AI game, assume $K_\delta \geq 0$}\label{alg_AI}
\eIf{
$
\theta > f(\theta, \delta) := \frac{\delta_a-\eta^*_\theta \delta}{(\Delta_R)^a(\delta_a-\alpha_R \eta^*_\theta)}
$
}{
$\widetilde{\theta} \gets \theta$
}{
\If{$K_\delta \geq 0$}{
$\widetilde{\theta} \gets \min\left\{\max\left\{\frac{-\kappa + \sqrt{K_\delta}}{2(\Delta_R)^a\alpha_R}, 1 - \frac{\delta(1-\alpha_F)}{\alpha_R} \right\} + \epsilon, 1\right\}$, \mbox{ for } $\epsilon > \max\left\{0, \theta - \frac{-\kappa + \sqrt{K_\delta}}{2(\Delta_R)^a\alpha_R}\right\}$
}}

\textbf{Choose:}

(i) $w \gets \frac{1}{c \alpha_R}
\left( 
\frac{1}{1-\mua} \max\left\{1, \frac{1}{(\Delta_R)^a \widetilde{\theta}}\right\} + \epsilon_1\right)$, where

\vspace{-2mm}
{\footnotesize
$$
\hspace{-1mm} 0 < \epsilon_1 < \min\left\{ \frac{1}{\delta_a}, \fone \right\} - \frac{1}{1-\mua}\max\left\{1, \frac{1}{(\Delta_R)^a\widetilde{\theta}} \right\},
$$}

% (ii) $\omega( \beta) \gets w^{1/b} \alpha_R^{(1-a)/b} \beta^{1/b}$,

(ii) $\eta \gets \overline{\eta} + \epsilon_2$, where $\epsilon_2 \in (0, \eta^*_{\widetilde{\theta}}-\overline{\eta}]$.
\end{algorithm}
\vspace{-0.8cm}
\begin{thm}\label{thrm_response1}
Consider the response function as in \eqref{eqn_response1}. Let $\theta \in \left(\max\left\{\alpha_F, \frac{\delta}{(\Delta_R)^a}\right\}, 1\right]$ and $\delta \in (\alpha_R, \theta)$. If the parameters satisfy the conditions in Algorithm \ref{alg_AI},
then, choosing $R, \gamma$ as in \eqref{eqn_R_gamma} and the warning mechanism as in \eqref{eqn_warning} for $w$ given in Algorithm \ref{alg_AI} leads to a game $\G(R, \gamma, \omega)$ such that:

\noindent (i) $\bmu_\eta$ is an AI-NE, for $\eta$ in Algorithm \ref{alg_AI}, and

\noindent (ii) $\bmu_{x_\eta}$ is the only other NE, with $x_\eta := \frac{p}{\gamma-1} + p(1-\mua-\eta) + \eta$, if $x_\eta > \eta_{\widetilde \theta}^*$.  
\eop
\end{thm}

\noindent The proof of above Theorem is in  Appendix (section \ref{appendix}). Thus, the OSN can design an AI game for any $(\theta, \delta)$ in the following:

\vspace{-4mm}
{\small
\begin{align}\label{eqn_feasible}
\hspace{-2mm}\RAI := \{(\theta, \delta) : \theta > f(\theta, \delta) \mbox{ or } \theta \leq f(\theta, \delta) \mbox{ with } K_\delta \geq 0\}.
\end{align}}Note that the above algorithm assumes the knowledge of user specific parameters, estimation of which is independent of linguistic barriers as discussed in the introduction.

We show in Lemma \ref{lemma_feasibility_w} that a feasible $w$ exists as per Algorithm \ref{alg_AI}.  It is important to note that the warning mechanism in \eqref{eqn_warning} is designed such that $\beta_F^\eta$ and $\beta_R^\eta$ correspond to $r = 1$ and $r < 1$ respectively, for some $\eta$. Such a design helps to ensure that $\beta_F^\eta \geq \theta_a$ and $\beta_R^\eta \leq \delta_a$. However, if the desired $\theta$ is small, then due to insufficient difference between $\theta$ and $\delta$, AI is not achievable. Then, we choose some $\widetilde{\theta} > \theta$ as in Algorithm \ref{alg_AI}. We formalise these ideas in the proof of Theorem \ref{thrm_response1}. In all, the OSN actually achieves 
$(\widetilde{\theta}, \delta)$-success at $\bmu_\eta$, for  $\widetilde{\theta} \geq \theta$ (see Lemma \ref{lemma_theta_tilde}).

% The design of an AI game is possible whenever there is non-negligible difference between $\alpha_F$ and $\alpha_R$, i.e., if the users are sufficiently smart enough to distinguish between $R$ and $F$-posts. This non-negligible desired difference is quantified in the specified conditions of Algorithm \ref{alg_AI}.  %At AI-NE, there are  $\eta$ and $1-\eta-\mua$ fraction of type $1$ and $2$ users respectively. 

The designed game $\G(R, \gamma, \omega)$ has a unique NE, which is AI if $x_\eta \leq \eta_{\widetilde \theta}^*$; else, there is another NE, $\bmu_{x_\eta}$. In the latter case, $x_\eta > \eta_{\widetilde \theta}^* \geq \eta$, and therefore, the performance of $\omega(\cdot)$ might degrade due to a larger proportion of type $1$ users at this NE. Next, we characterise  $\bmu_{x_\eta}$ (see proof in Appendix in section \ref{appendix}).
\begin{thm}\label{thrm_perf_x_eta}
Define $x_F := \xdoubleF$. Under the hypothesis of Theorem \ref{thrm_response1}, if $x_\eta > \eta^*_{\widetilde{\theta}}$, then:
\[
\beta_R^{x_\eta} \leq \delta_a, \ \beta_F^{x_\eta} \geq 
\begin{cases}
 \frac{1}{cw\alpha_R (\Delta_R)^a}, &\mbox{if } x_\eta \in (\eta^*_{\widetilde{\theta}}, x_F]. \\
 \alpha_F(1-\mua), &\mbox{if } x_\eta \in (x_F, 1-\mua). \mbox{\eop}
\end{cases}
\]
\end{thm}
It is easily verifiable that $\beta_F^{x} = \widetilde{\theta}(1-\mua)$ for $x = \eta^*_{\widetilde{\theta}}$. Now, as one might expect, we prove in \eqref{eqn_betaF}, \eqref{eqn_betaR} that the proportion of fake tags should decrease with an increase in type $1$ users for any $u$-post. In view of this, Theorem \ref{thrm_perf_x_eta} states that if the proportion of type $1$ users, \hide{$x_\eta \leq \eta^*_{\widetilde{\theta}}$, then, the second NE ($\bmu_{x_\eta}$) achieves $(\widetilde{\theta}, \delta)$-success, and hence $\td$. Then we have a game with both the NEs as  AI. %Further, observe that $\beta_F^\eta > \beta_F^{x_\eta} \geq \widetilde{\theta}(1-\mua)$ and $\beta_R^{x_\eta} < \beta_R^\eta \leq \delta_a$. Thus, the OSN might prefer $\bmu_{x_\eta}$ as it leads to lesser users fake-tagging the $R$-posts, while still achieving the desired level of identification of $F$-posts.
Even when  }$x_\eta > \eta^*_{\widetilde{\theta}}$, the proposed warning mechanism does not degrade the quality of the tagging process for $R$-post, as from Theorem \ref{thrm_perf_x_eta}, we have $\beta_R^{x_\eta} \leq \delta_a$. However, the users can not identify the $F$-post up to $\theta_a(\bmu_{x_\eta})$-level. Theorem \ref{thrm_perf_x_eta} provides the worst performance of the warning for tagging of $F$-posts.

We next numerically comment upon a more detailed performance of $\bmu_{x_\eta}$ for $F$-post using the normalized degradation metric, ${\cal P} := \nicefrac{(\theta_a(\bmu_{x_\eta}) - \beta_F^{x_\eta})100}{\theta_a}$. Towards this, we consider a large number of samples/configurations of system parameters chosen randomly and independently from some appropriate uniform distributions to obtain the fraction of configurations that achieve AI; we also obtain  the fraction of configurations that have   ${\cal P} < 10\%$. We pick  $\alpha_R \sim  U(0.25, 0.3),\ \mua \sim U(0, 0.2),\ a \sim U(2,3), \ p \sim U(0, 0.5) \mbox{ and } \delta = \alpha_R + 0.01$. 
Define $d$ as the normalised difference  between the innate capacity of users to identify $F, R$-posts, i.e., $d = \nicefrac{(\alpha_F - \alpha_R)}{\alpha_F}$.
We generate $10000$ samples for different values of $d$. %Here, each sample can be viewed as a configuration for a separate post; since users react differently to various posts, the parameter $a$ of the response function \eqref{eqn_response1} is not fixed; some posts might be supported by less/more adversarial users, hence $\mua$ also varies. 
%
%
% \begin{table}[htbp]
% \caption{Performance of $\bmu_{x_\eta}$ for $F$-post}
% \begin{tabular}{|c|cccccc|}
% \hline
%             & \multicolumn{6}{c|}{$\%$ of number of samples with ${\cal P} < 10\%$}                                                                                              \\ \hline
% \backslashbox{$\theta$}{$d$} & \multicolumn{1}{c|}{$0.08$}  & \multicolumn{1}{c|}{$0.12$}  & \multicolumn{1}{c|}{$0.16$}  & \multicolumn{1}{c|}{$0.2$}   & \multicolumn{1}{c|}{$0.24$}  & $0.28$  \\ \hline
% $0.75$      & \multicolumn{1}{c|}{$21.16$} & \multicolumn{1}{c|}{$46.39$} & \multicolumn{1}{c|}{$52.11$} & \multicolumn{1}{c|}{$53.67$} & \multicolumn{1}{c|}{$55.86$} & $58.57$ \\ \hline
% $0.85$      & \multicolumn{1}{c|}{$33.73$} & \multicolumn{1}{c|}{$40.97$} & \multicolumn{1}{c|}{$42.23$} & \multicolumn{1}{c|}{$43.48$} & \multicolumn{1}{c|}{$44.87$} & $46.41$ \\ \hline
% \end{tabular}
% \label{table}
% \end{table}

Now, say that the OSN demands $\theta = 0.75$. Then, it can always design an AI game (for all random configurations) if $d \geq 0.01$. %; moreover, it can always design a full AI game  for any $d\geq 0.08$. 
Further, $21.16\%$ of samples have ${\cal P} < 10\%$ for $d = 0.08$, which gradually increases to $58.57\%$ as users get smarter, $d = 0.28$. Thus, if the OSN aims to achieve higher performance with respect to $\bmu_{x_\eta}$ as well, then it requires users to be slightly more intelligent (higher $d$).

\vspace{-2mm}
\section{Conclusion}
OSNs are flooded with fake posts and several techniques have been proposed to detect the same. A significant fraction of them depend upon crowd signals; however, none focusses on the limited willingness of the crowd to participate. We filled this gap by formulating an appropriate (mean-field) participation game where the users are encouraged to provide their responses (fake/real) for each post via a simple reward-based scheme. Further,  our algorithm ensures minimal wrong judgement of real/authentic posts and maximal actuality identification of the fake ones.

We proposed a simple warning mechanism for the polynomial response function of the users. Our mechanism is robust against adversarial users, independent of language barriers and continually guides the users to make more informed decisions by utilizing the warning signals shared by the OSN. % and their innate ability to judge the veracity of the post. 
Under our design, the resultant game always has a Nash Equilibrium (NE), which meets the desired objective. We also identify the condition in which there exists another NE; it achieves the desired identification level for real posts, but fails for achieving the desired level for fake posts. % succeeds in meeting the objective of the OSN. 

 % \vspace{-3mm}

\bibliographystyle{IEEEtran}

%%%%%%%%%%%%%%%%%%%%%%%%%%%%%%%%%%%%%%%%%%%%%%%%%%%%%%%%%%%%%%%%%%%%%%%%%%%%%%%%%%%%%%
\section{Appendix}\label{appendix}
\noindent \underline{\textbf{Proof of Lemma \ref{lemma_beta}:} }
The proof of this Lemma follows from \cite[Theorem 2.1, pp. 127]{kushner2003stochastic} under \textbf{(A)}, if further assumptions (A2.1)-(A2.5) of the cited Theorem hold, which we prove next. At first, observe $\sup_n E[L_{u, n}^2] < 2 < \infty$ (since from \eqref{eqn_dynamics_beta}, $\beta_{u, n} \leq 1$) for each $u$. Further, from \eqref{eqn_cond_exp}, there is no bias term as in the cited Theorem and $g_u(\cdot)$ is Lipschitz continuous. Lastly, $\sum_{i\geq 1} \epsilon_i^2 < \infty$ (for $\epsilon_i := 1/i$). \eop

% \vspace{2mm}

% \noindent \underline{\textbf{Proof of Lemma \ref{lemma_R_gamma}:}} Consider any $R, \gamma$ as in \eqref{eqn_R_gamma}. Then, we have $U(1, \bmu), U(2, \bmu) > U(0, \bmu)$ for any $\bmu$. Thus, for any $\bmu$ to be a NE, $\S(\bmu) \subset \{1, 2\}$. This implies that any $\bmu$ with $\mu_0 > 0$ can not be a NE. We are then left with $\bmu_x := (0, x, 1-x-\mua)$ for any $x \in [0, 1-\mua]$ as the possible choices of NE for f-$\G(R, \gamma)$. 

% Observe that $\bmu_\eta$ is a NE as for given $R, \gamma$, we have $U(1, \bmu_\eta) = U(2, \bmu_\eta)$
% % 
% % \vspace{-4mm}
% % {\small $$
% %     \cp + \frac{R}{\gamma - (\gamma-1)(\eta+\mua)} = \cp - \ce + \frac{R\gamma}{\gamma - (\gamma-1)(\eta+\mua)},
% % $$}
% (see definition \ref{defn_NE_MFG}). However, 
% % We have, $R(\gamma-1) = (\gamma - (\gamma-1)(\eta+\mua))\ce$, and thus, one can easily verify that the above equality is true for the given fictitious game. 
% for any $x \in (\eta, 1-\mua)$ or $x \in (0, \eta)$, we have $U(1, \bmu_x) < U(2, \bmu_x)$ and $U(1, \bmu) > U(2, \bmu)$ respectively; thus, any such $\bmu_x$ contradicts definition \ref{defn_NE_MFG} and hence, can not be a NE.

% Further, $U(1, \bmu_{1-\mua}) < U(2, \bmu_{1-\mua})$ (as $\gamma > 1$); thus, $\S(\bmu_{1-\mua}) \not\subset \mbox{argmax}_{s} U(s, \bmu_{1-\mua})$ which implies that $\bmu_{1-\mua}$ is not a NE. One can analogously prove that $\bmu_0$ is also not a NE for f-$\G(R,\gamma)$. \eop

\vspace{2mm}

\noindent \underline{\textbf{Proof of Theorem \ref{thrm_AI}:} } Consider $u \in \{R, F\}$. 
By hypothesis (B.i), (B.ii), $g_u(\beta_u^\eta) = 0$ for some $\beta_u^\eta$. Further, by hypothesis (B.iv), i.e., local stability, $g_u(\beta_u) > 0$ for all $\beta_u \in (\beta_u^\eta-\epsilon, \beta_u^\eta)$,  and $g_u(\beta_u) < 0$ for all $\beta_u \in (\beta_u^\eta, \beta_u^\eta + \epsilon)$, for some $\epsilon > 0$. Since $g_u(\cdot)$ is a continuous function with unique zero (see hypothesis (B.iii)), $g_u(\beta_u) > 0$ for all $\beta_u \in [0, \beta_u^\eta)$ and $g_u(\beta_u) < 0$ for all $\beta_u \in (\beta_u^\eta, 1]$. Thus, $t \mapsto \beta_u(t)$ is strictly increasing and decreasing, if $\beta_u(0) = \beta_u \in [0, \beta_u^\eta)$ and $(\beta_u^\eta, 1]$ respectively.  This implies that the assumption \textbf{(A)} of Lemma \ref{lemma_beta} is satisfied with $\cA_u = \{\beta_u^\eta\}$ and $\cD_u = [0,1]$ for each $u$. Thus, by Lemma \ref{lemma_beta}, $\beta_{u, k} \to \beta_u^\eta$ w.p. $1$. % and hence, $P_{\bmu_\eta}(S) = 1$ (under hypotheses (i), (ii)). 

% In view of the above arguments and utility functions as in \eqref{eqn_util} and \eqref{eqn_util_fictitious},  $\bmu_\eta$ is the unique NE of the f-$\G(R, \gamma)$ (by Lemma \ref{lemma_R_gamma}) and a AI-NE for the original game $\G(R, \gamma, \omega)$. This completes (a).
For the given $R, \gamma$, we now prove (a)-(c) for the game $\G(R, \gamma, \omega)$. By hypothesis (B.i), (B.ii) and above arguments, $P_{\bmu_\eta}(S; \theta, \delta) = 1$ (see \eqref{eqn_simplified_prob_success}). Also, from \eqref{eqn_util} and \eqref{eqn_R_gamma}, {\small$U(1, \bmu_\eta) = U(2, \bmu_\eta) > U(0, \bmu_\eta)$}. Thus, $\S(\bmu_\eta) = \mbox{arg max}_s U(s, \bmu)$, which by definition \ref{defn_NE_MFG} implies that $\bmu_\eta$ is an AI-NE. Hence, part (a). 

Now, if possible, let $\bmu$ be another NE such that $P_{\bmu}(S; \theta, \delta) =: q \in [0,1]$. By \eqref{eqn_util}, {\small$U(1, \bmu) \geq U(0, \bmu)$}, as $Q_p \ge Q_{np}$. 

First consider the case with $U(1, \bmu) > U(0, \bmu)$. Thus, $0 \notin \S(\bmu)$, and hence, any $\bmu$ with $\bmu_0 > 0$ can not be a NE for this case (see definition \ref{defn_NE_MFG}). For the rest, we divide the proof in two sub-cases:

$\bullet$ If $\S(\bmu) = \{1\}$, then $\bmu = \bmu_{1-\mua}$. By Lemma \ref{lemma_beta_only_type1}, $\beta_F(\bmu) < \theta_a(\bmu)$ and $\beta_R(\bmu) < \delta_a(\bmu)$. Thus, $\bmu$ can be a NE (if at all) when $q = 1-p$. Further, $\bmu$ being a NE implies that $U(1, \bmu) \geq U(2, \bmu)$, with utility function as in \eqref{eqn_util}. That is, $R, \gamma$ should satisfy the following relation:
$$
    R(1-p) (\gamma-1) \leq \ce, \mbox{ i.e., } g(\gamma) \leq 1, \mbox{ where}
$$
$g(\gamma) := (\gamma-(\gamma-1)(\eta+\mua)) (1-p)$. Observe, $\gamma \mapsto g(\gamma)$ is increasing and $g(\underline{\gamma}(\eta)) = 1$; thus, $g(\gamma) > 1$ for $\gamma$ given in hypothesis. This contradicts $\bmu$ being a  NE. 

$\bullet$ If $\S(\bmu) = \{2\}$ or $\{1, 2\}$, then $\bmu = \bmu_x$ for some $x\in [0, 1-\mua) - \{\eta\}$. One can show that $U(1, \bmu) > U(2, \bmu)$ (for any $q \in [0,1]$). Thus, $\S(\bmu) \not\subset \mbox{argmax}_s U(s, \bmu) = \{1\}$, which contradicts $\bmu$ being a  NE. 

% Here, we have:
% \begin{align*}
% U(1, \bmu) - U(2,\bmu) &= \frac{R(1-\gamma)q}{\mua + \gamma(1-\mua)} + \ce\\
%     &\hspace{-15mm}= \ce\left(1-q\left(1-\frac{(\gamma-1)\eta}{\gamma-(\gamma-1)\mua}\right)\right) > 0.
% \end{align*}
% Thus, $\S(\bmu) \not\subset \{1\} = \mbox{argmax}_s U(s, \bmu)$, implying that $\bmu = (0, 0, 1-\mua)$ also can not be a NE, see definition \ref{defn_NE_MFG}.

% $\bullet$ If $\S(\bmu) = \{1, 2\}$, then $\bmu = \bmu_x$ for $x \in (0, 1-\mua)$. In part (b), we focus on $x \in (0, \eta)$. If $x < \eta$, then one can show that $U(1, \bmu) > U(2, \bmu)$ (for any $q \in [0,1]$). Thus, $\S(\bmu) \not\subset \mbox{argmax}_s u(s, \bmu) = \{1\}$, which contradicts definition \ref{defn_NE_MFG} and implies that $\bmu$ can not be a NE. 
% If $x \in (\eta, \eta^*]$, then, under hypothesis (v),  $P_{\bmu}(S)=1$. By Lemma \ref{lemma_R_gamma}, such a $\bmu$ can not be a NE. For $x \in (\eta^*, 1-\mua)$, $\bmu_x$ can be a NE only if $u(1, \bmu_x) \neq u(2, \bmu_x)$ (see definition \ref{defn_NE_MFG}).

Lastly, consider the case where $U(1, \bmu) = U(0, \bmu)$; this is possible only when $q = 0$ and $\cnp = \cp$. Clearly,  $\cp > \cp - \ce = U(2, \bmu)$. 
Thus, $2 \notin \S(\bmu)$. If $\bmu = \bmu_{1-\mua}$, from Lemma \ref{lemma_beta_only_type1}, $\beta_R(\bmu) < \delta_a(\bmu)$. This implies, $q \neq 0$, which is a contradiction. Another possibility for $\bmu$ is $(1-\mua, 0, 0)$ which can not be NE as $\beta_F(\bmu) = 0 = \theta_a(\bmu)$, leading to $q \geq p \neq 0$. The last possibility for $\bmu$ is $(x, 1-x-\mua, 0)$ for any $x \in (0, 1-\mua)$, for which we have:
\begin{align*}
\beta_R(\bmu) &= \alpha_R\left(\frac{1-x-\mua}{1-x} \right) < \delta \left(\frac{1-x-\mua}{1-x} \right) = \delta_a(\bmu).
\end{align*}Thus, $q \neq 0$; hence, any such $\bmu$ also can not be a NE.  \eop

\begin{lemma}\label{lemma_beta_only_type1}
For $\bmu = \bmu_{1-\mua}$,  $\beta_{u, k}(\bmu) \to \alpha_u(1-\mua)$ w.p. $1$, as $k \to \infty$, for each $u$. 
\end{lemma}
\begin{proof}
From \eqref{eqn_dynamics_beta}, $\beta_{u, k}$ can be re-written as follows:
\begin{align*}
    \beta_{u, k}(\bmu) &= \frac{\sum_{i=1}^k 1_{\{\mbox{tag for $u$-post} = F\}}}{k}.
\end{align*}
Thus, $\beta_{u, k}(\bmu) \to \alpha_u (1-\mua)$ w.p. $1$, as $k \to \infty$, by strong law of large numbers, and \eqref{eqn_cond_exp}.
\end{proof}

\noindent \underline{\textbf{Proof of Theorem \ref{thrm_response1}:}} The proof is in $3$ steps:

\noindent (a) $\bmu_\eta$ is an AI-NE such that $\beta_F^\eta \geq \widetilde{\theta}(1-\mua) \geq \theta(1-\mua)$ and $\beta_R^\eta < \delta_a = \delta_a(\bmu_\eta)$,

\noindent (b) by Theorem \ref{thrm_AI}, any $\bmu$ with $\mu_0 > 0$ or any $\bmu_x$ for $x \in [0, \eta) \cup \{1-\mua\}$ is not a NE,

\noindent (c) $\bmu_{x_\eta}$ can be the only other NE, if at all $x_\eta > \eta_{\widetilde \theta}^*$ and $\beta_F^{x_\eta} < \theta(1-\mua)$. % $\beta_R^x < \delta_a$ for $x \in (\eta, 1-\mua)$; any $\bmu_x$ for $x \in (\eta, \eta_{\widetilde \theta}^*]$ is not a NE, and 

% \noindent (d) for $x \in (\eta_{\widetilde \theta}^*, 1-\mua)$, $\bmu_x$ is a NE only for $x = x_\eta$ only if $x_\eta > \eta_{\widetilde \theta}^*$; $\beta_R^{x_\eta} < \delta_a$ and $\beta_F^{x_\eta} < \widetilde{\theta} (1-\mua)$.

Define $x_F := \frac{1-\mua - \frac{1}{cw\alpha_R(\Delta_R)^a}}{1-\alpha_F}$.
Define $\rho_F(x) := 1-(1-x-\mua)cw\alpha_R(\Delta_R)^a$ for $x \in (0,1)$. Then, $\rho_F(x) = 0$ for $x =  x_F(1-\alpha_F)$. Also, $\rho_F(x)$ is increasing in $x$. Therefore, $\rho_F(x) < 0$ for $x < x_F(1-\alpha_F)$ and $\rho_F(x) > 0$ for $x >  x_F(1-\alpha_F)$. 

Next define $\overline{\rho}_F(x) := \alpha_F x + 1 - x - \mua$ for $x \in (0,1)$. Then, $\overline{\rho}_F(x) = \frac{1}{cw\alpha_R(\Delta_R)^a}$ for $x = x_F$ and $\overline{\rho}_F(x)$ is decreasing in $x$. Therefore, $\overline{\rho}_F(x) \in {\cal R}_F:= \left\{y: y < \frac{1}{cw\alpha_R(\Delta_F)^a}\right\}$ for all $x > x_F$ and $\overline{\rho}_F(x) \in {\cal R}_F^c$ for all $x \leq x_F$.% This is equivalent to $\frac{\alpha_Fx}{\rho_F(x)} < \frac{1}{cw\alpha_R\Delta^a}$ for $x > x_F$ and $\frac{\alpha_Fx}{\rho_F(x)} \geq \frac{1}{cw\alpha_R\Delta^a}$ for $x \leq x_F$. 

In all, by above, $\overline{\rho}_F(x) \in {\cal R}_F^c$ for $x \leq x_F(1-\alpha_F)$. If not, $\rho_F(x) > 0$ for all $x > x_F(1-\alpha_F)$, and then by Lemma \ref{lemma_beta_traj}, both $\overline{\rho}_F(x)$ and $\frac{\alpha_F x}{\rho_F(x)}$ are in ${\cal R}_F^c$ for $x \in (x_F(1-\alpha_F), x_F]$; both $\overline{\rho}_F(x)$ and $\frac{\alpha_F x}{\rho_F(x)}$ are in ${\cal R}_F$ for $x > x_F$. Further, by Lemma \ref{lemma_beta_traj}, $\beta_F^x$ (the attractor of ODE \eqref{eqn_ode}) is given by:
\begin{align}\label{eqn_betaF}
\beta_F^x = 
\begin{cases}
 \overline{\rho}_F(x) \mbox{ if } x \in (0, x_F], \\
 \frac{\alpha_F x}{\rho_F(x)} \mbox{ if } x \in (x_F, 1).
\end{cases}
\end{align}

Similarly, again by Lemma \ref{lemma_beta_traj}, one can show that $\beta_R^x$ is:
\begin{align}\label{eqn_betaR}
\beta_R^x = 
\begin{cases}
 \overline{\rho}_R(x) \mbox{ if } x \in (0, x_R], \\
 \frac{\alpha_R x}{\rho_R(x)} \mbox{ if } x \in (x_R, 1),
\end{cases}
\end{align}
for $\overline{\rho}_R(x) := \alpha_R x + 1 - x - \mua$, $\rho_R(x) := 1-(1-x-\mua)cw\alpha_R$ and $x_R := \frac{1-\mua-\frac{1}{cw\alpha_R}}{1-\alpha_R}$.

% It is also easy to verify that $\beta_u^x$ is strictly decreasing in $x$ for each $u \in \{R, F\}$, by simply checking that $\frac{\partial \beta_u^x}{\partial x} < 0$.

 Observe that by the choice of $w$ as in Algorithm \ref{alg_AI}, 
\begin{align}\label{eqn_bound_cwalphaR}
    cw\alpha_R > \frac{1}{(\Delta_R)^a \widetilde{\theta} (1-\mua)}.
\end{align}
Thus, from \eqref{eqn_notations}, $\eta^*_{\widetilde{\theta}} < x_F$. Consider any $x \leq \eta^*_{\widetilde{\theta}}$. By \eqref{eqn_betaF}, 
$\beta_F^x = \overline{\rho}_F(x)$. Since $\beta_F^x$ strictly decreases with $x$, therefore, $\beta_F^x \geq \beta_F^{\eta^*_{\widetilde{\theta}}} = \widetilde{\theta}(1-\mua) \geq \theta(1-\mua)$, by Lemma \ref{lemma_theta_tilde}. 

Observe that $\eta = \overline{\eta} + \epsilon_2 \leq \eta^*_{\widetilde{\theta}}$ (see claim $1$ at the end of the proof for details), thus proving conditions (B.ii), (B.iii) and (B.iv) for $u = F$ of Theorem \ref{thrm_AI} for $\bmu_\eta$.

% \begin{figure}[htbp]
%     \centering
%     \includegraphics[scale=0.35]{regime1.pdf}
%     \caption{Analysis of $\beta_u^x$, for $u \in \{R, F\}$}
%     \label{fig_regime1}
% \end{figure}

Again by the choice of $w, \epsilon_2$ and $\delta>\alpha_R$, $\eta - x_R > \overline{\eta}-x_R > 0$ (see claim $2$ at the end of the proof). Therefore, by \eqref{eqn_betaR}, $\beta_R^\eta = \frac{\alpha_R \eta}{{\rho}_R(\eta)}$. 
By the choice of $w$ and since $\beta_R^x$ strictly decreases as $x$ increases, we have, $\beta_R^\eta < \beta_R^{x_R} = 1/(cw \alpha_R) < \delta_a$. 
This proves conditions (B.i), (B.iii) and (B.iv) for $u = R$ of Theorem \ref{thrm_AI}  for $\bmu_\eta$. 

In all, by Theorem \ref{thrm_AI}, $\G(R, \gamma, \omega)$ is an AI game with $\bmu_\eta$ as a NE such that it achieves $(\widetilde{\theta}, \delta)$-success (i.e., $\td$ as $\widetilde{\theta} \geq \theta$ by Lemma \ref{lemma_theta_tilde}); further, any $\bmu_x$ for $x\in [0, \eta) \cup \{1-\mua\}$  and any $\bmu$ with $\mu_0 > 0$ can not be a NE, by Theorem \ref{thrm_AI}. This complete steps (a) and (b).

Consider any $x \in (\eta, 1-\mua)$. Since $\beta_R^x$ decreases in $x$, $\beta_R^x < \beta_R^\eta < \delta_a$. This proves (b). Recall $\beta_F^x \geq \theta(1-\mua)$ for $x \in (\eta, \eta^*_{\widetilde{\theta}}]$. Thus:
$$
P_{\bmu_x}(S; \widetilde{\theta}, \delta) = 1 \mbox{ for each } x \in (\eta, \eta^*_{\widetilde{\theta}}].
$$
For the given $R, \gamma$ and chosen $x$, one can show that $U(1, \bmu_x) < U(2, \bmu_x)$. Thus, $\S(\bmu_x) = \{1,2\} \not\subset \mbox{argmax}_s u(s, \bmu) = \{2\}$; under definition \ref{defn_NE_MFG}, $\bmu_x$ is not a NE. In fact, if $P_{\bmu_x}(S;\theta, \delta) = 1$ for some $x \in (\eta^*_{\widetilde{\theta}}, 1-\mua)$, then again using above arguments, one can show that $\bmu_x$ is not a NE. 

Recall $\beta_R^x < \delta_a$ for each $x \in (\eta, 1-\mua)$. Further by definition of $x_\eta$, $U(1, \bmu_x) = U(2, \bmu_x)$ only for $x = x_\eta$ with $
P_{\bmu_x}(S; \theta, \delta) = 1-p$; further, by \eqref{eqn_util}, $x_\eta$ is the only such possible $x$. Thus, by definition \ref{defn_NE_MFG}, $\bmu_{x_\eta}$ is a NE, but not AI-NE, if at all $x_\eta > \eta^*_{\widetilde{\theta}}$ and $\beta_F^{x_\eta} < \theta(1-\mua)$. This completes step (c). 
% Thus, $\S(\bmu_{x_\eta}) = \{1, 2\} = \mbox{argmax}_s U(s, \bmu_{x_\eta})$.

Now, we will prove the sub-claims made above. 

\noindent \underline{Claim 1:} $\eta^*_{\widetilde{\theta}}>\overline{\eta}$. Let us consider the difference:
\begin{align*}
        \eta^*_{\widetilde{\theta}} - \overline{\eta} &=  \eta^*_{\widetilde{\theta}} - \frac{\delta_a((1-\mua)cw\alpha_R - 1)}{cw\alpha_R \delta_a - \alpha_R}\\
        &= \frac{cw\alpha_R   \delta_a (\eta^*_{\widetilde{\theta}} - (1-\mua))  - \eta^*_{\widetilde{\theta}} \alpha_R + \delta_a }{cw\alpha_R \delta_a - \alpha_R}\\
        &=  \frac{-(1 - \mua -  \eta^*_{\widetilde{\theta}})  \delta_a\left( cw\alpha_R   - \frac{- \eta^*_{\widetilde{\theta}} \alpha_R + \delta_a}{\delta_a (1 - \mua  - \eta^*_{\widetilde{\theta}} )} \right) }{cw\alpha_R \delta_a - \alpha_R}
\end{align*}
Now, $1 - \mua >  \eta^*_{\widetilde{\theta}} $. Further, by the choice of $w$, $cw\alpha_R < \fone$, 
which implies the numerator in $\eta^*_{\widetilde{\theta}} - \overline{\eta}$ is strictly positive. Furthermore, we have $cw\alpha_R > \frac{1}{1-\mua} > \frac{\alpha_R}{\delta(1-\mua)}$. This implies that the denominator  in $\eta^*_{\widetilde{\theta}} - \overline{\eta}$  is strictly positive. Therefore, $\eta^*_{\widetilde{\theta}}>\overline{\eta}$.

\noindent \underline{Claim 2:} $\eta > x_R$

\vspace{-4mm}
{\small
\begin{align*}
    \eta - x_R &> \overline{\eta} - x_R \\
    &\hspace{-6mm}= \frac{\delta_a((1-\mua)cw\alpha_R - 1)}{cw\alpha_R \delta_a - \alpha_R} - \frac{1}{1-\alpha_R} \left( 1-\mua-\frac{1}{cw\alpha_R}\right)\\
    &\hspace{-6mm}= ((1-\mua)cw\alpha_R - 1) \left( \frac{\delta_a }{cw\alpha_R \delta_a - \alpha_R} - \frac{1}{cw\alpha_R(1-\alpha_R)} \right)\\
    &\hspace{-6mm}= \frac{((1-\mua)cw\alpha_R-1) (1-cw\alpha_R\delta_a)}{cw\alpha_R(1-\alpha_R) (cw\alpha_R\delta_a - \alpha_R)} > 0 \mbox{ (by choice of $w$)}.  
\end{align*}} \eop

\begin{lemma}\label{lemma_beta_traj}
    Define $\overline{\rho}_u := \alpha_u\eta + 1 - \eta - \eta_a$ and $\rho_u := 1 - (1-\eta-\eta_a)cw\alpha_R(\Delta_u)^a$ for $\eta \in (0, 1-\eta_a)$ and $u \in \{R, F\}$. Consider the regime, ${\cal R}_u := \left\{x: x < \frac{1}{cw\alpha_R(\Delta_u)^a}\right\}$. Then, for the response function given in \eqref{eqn_response1}, the following statements are true:
    \begin{enumerate}
        \item if $\rho_u \leq 0$, then $\overline{\rho}_u \in {\cal R}_u^c$, and the attractors of ODE \eqref{eqn_ode}, ${\cal A}_u = \{\overline{\rho}_u\}$; 
        \item if $\rho_u > 0$, then $\overline{\rho}_u \in {\cal R}_u^c$ if and only if $\frac{\alpha_u\eta}{\rho_u} \in {\cal R}_u^c$. Further, if $\overline{\rho}_u \in {\cal R}_u$ then ${\cal A}_u = \{\frac{\alpha_u\eta}{\rho_u}\}$, while if $\overline{\rho}_u \in {\cal R}^c_u$ then ${\cal A}_u = \{\overline{\rho}_u\}$.
    \end{enumerate}
\end{lemma}
\begin{proof}
    At first, let $\rho_u \leq 0$. Then, by definition of $\rho_u$, $1-\eta-\eta_a \geq \frac{1}{cw\alpha_R(\Delta_u)^a}$. Since $\eta > 0$, therefore, 
    $$
    \overline{\rho}_u \geq \alpha_u \eta + \frac{1}{cw\alpha_R(\Delta_u)^a} > \frac{1}{cw\alpha_R(\Delta_u)^a} \implies \overline{\rho}_u \in {\cal R}_u^c.
    $$ 
    The ODE \eqref{eqn_ode} can be written in the simplified form as follows (recall $r(\alpha_u, \omega(\beta_u)) = \min\{1, cw\alpha_R(\Delta_u)^a\beta_u\}$):
\begin{align}\label{eqn_simplified_ODE}
\dot{\beta_u} = 
\begin{cases}
 \overline{\rho}_u- \beta_u, \mbox{ if } r(\alpha_u, \omega(\beta_u)) = 1, \mbox{ i.e., } \beta_u \in {\cal R}_u^c, \\
 \alpha_u \eta - \rho_u\beta_u, \mbox{ if } r(\alpha_u, \omega(\beta_u)) < 1, \mbox{ i.e., } \beta_u \in {\cal R}_u.
\end{cases}
\end{align}
Clearly, the RHS of the above ODE is piecewise linear, and hence the solution $\beta_u(\cdot)$ exists.

Now, say $\beta_u(0) \in {\cal R}_u$. Then, $\dot{\beta_u} > 0$, thus, $\beta_u(t)$ increases with $t$. This implies the existence of $\tau <\infty$ such that $cw\alpha_R(\Delta_u)^a\beta_u(\tau) = 1$. Then, the solution of the ODE for all $t\geq \tau$ is:
\begin{align}\label{eqn_beta_sol}
    \beta_u(t) = \overline{\rho}_u + e^{-t+\tau}(\beta_u(\tau) - \overline{\rho}_u).
\end{align}
The above solution holds for all $t\geq \tau$ as $r(\alpha_u, \omega(\beta_u(t))) = 1$ for all $t\geq \tau$; towards this, observe that $\beta_u(\tau) \leq \overline{\rho}_u$ (since $\overline{\rho}_u \in {\cal R}_u^c$), therefore, $t \mapsto \beta_u$ is an increasing function. Hence, from \eqref{eqn_beta_sol}, $\beta_u(t) \to \overline{\rho}_u$ as $t \to \infty$. On the contrary if $\beta_u(0) \in {\cal R}_u^c$, i.e., $r(\alpha_u, \omega(\beta_u(0))) = 1$, then for all $t\geq 0$ (check $r(\alpha_u, \omega(\beta_u(t))) = 1$ for all $t\geq 0$):
\begin{align}\label{eqn_beta_sol_2}
    \beta_u(t) = \overline{\rho}_u + e^{-t}(\beta_u(0) - \overline{\rho}_u).
\end{align}
From above, $\beta_u(t) \to \overline{\rho}_u$.

Now, let $\rho_u > 0$. Then, by definitions, we have:
\begin{align*}
    \frac{\alpha_u\eta}{\rho_u} \in {\cal R}_u^c &\iff cw\alpha_R(\Delta_u)^a  \frac{\alpha_u\eta}{\rho_u} \geq 1\\
    &\iff cw\alpha_R(\Delta_u)^a\alpha_u\eta \geq \rho_u\\
    &\iff cw\alpha_R(\Delta_u)^a\overline{\rho}_u \geq 1 \iff \overline{\rho}_u \in {\cal R}_u^c.
\end{align*}
We will now derive ${\cal A}_u$ for the case when $\overline{\rho}_u \in {\cal R}_u$, and ${\cal A}_u$ can be derived analogously for the complementary case. As before, say $\beta_u(0) \in {\cal R}_u$. Then initially the $\beta_u$-ODE is:
$$
    \dot{\beta_u} = \alpha_u \eta - \rho_u \beta_u.
$$
Thus, the solution of the above ODE is:
\begin{align*}
    \beta_u(t) = \frac{\alpha_u\eta}{\rho_u} + e^{-\rho_u t}\left(\beta_u(0) - \frac{\alpha_u\eta}{\rho_u} \right) \mbox{ for all } t \geq 0.
\end{align*}
Clearly, $\beta_u(t) \to \frac{\alpha_u\eta}{\rho_u}$.
If $\beta_u(0) \in {\cal R}_u^c$, then as previously, the ODE solution is given by \eqref{eqn_beta_sol_2} for all $t < \tau$, where $\tau := \inf\{t : r(\alpha_u, \omega(\beta_u(t))) < 1\}$. For $t \geq \tau$, the solution is:
$$
    \beta_u(t) = \frac{\alpha_u\eta}{\rho_u} + e^{-\rho_u(t-\tau)}\left( \beta_u(\tau) - \frac{\alpha_u\eta}{\rho_u} \right).
$$
Then, $\beta_u(t) \to \frac{\alpha_u\eta}{\rho_u}$.
\end{proof}

\noindent \underline{\textbf{Proof of Theorem \ref{thrm_perf_x_eta}:}} From the definition of $x_\eta$, note that $x_\eta > \eta$. From  \eqref{eqn_betaR}, $\beta_R^x$ is decreasing in $x$, and thus, by (a) in the proof of Theorem \ref{thrm_response1}, $\beta_R^{x_\eta} <  \beta_R^\eta \leq \delta_a$. By \eqref{eqn_betaR}:

% \noindent $\bullet$ if $x_\eta \leq \eta^*_{\widetilde{\theta}}$, $\beta_F^{x_\eta} \geq \beta_F^{\eta^*_{\widetilde{\theta}}} = \theta_a$,

\noindent $\bullet$ if $x_\eta \in (\eta^*_{\widetilde{\theta}}, x_F]$, $\beta_F^{x_\eta} \geq \beta_F^{x_F} = \overline{\rho}_F(x_F) = \frac{1}{cw\alpha_R (\Delta_R)^a}$;

\noindent $\bullet$ if $x_\eta \in (x_F, 1-\mua)$, $\beta_F^{x_\eta} \geq \beta_F^{1-\mua} = \alpha_F(1-\mua)$. \eop

\begin{lemma}\label{lemma_theta_tilde}
    For notations as in Algorithm \ref{alg_AI} when $K_\delta \geq 0$, ${\widetilde \theta} \geq \theta$.
\end{lemma}
\begin{proof}
    If $\theta > f(\theta, \delta)$, then ${\widetilde \theta} = \theta$, and we are done. Else, define, $g(x) := x - f(x, \delta)$ for $x \in \mathbb{R}$; observe $g(\theta) \leq 0$. Using the definitions in \eqref{eqn_notations}, re-write $g(x)$ as:
    \begin{align*}
        g(x) &= \left( \frac{1-\mua}{1-\alpha_F}\right) \frac{p(x)}{t(x)}, \mbox{ where}\\
        t(x) &:= (\Delta_R)^a\left(\delta_a - \frac{(1-x)(1-\mua)\alpha_R}{1-\alpha_F} \right) \mbox{ and}\\
        p(x) &:= Ax^2+\kappa x+C, \mbox{ for } A := (\Delta_R)^a \alpha_R \mbox{ and } C := \delta \alpha_F.
    \end{align*}
    Observe that $t(\cdot)$ is strictly increasing and $t(\theta^*) = 0$ for $\theta^* := 1 - \frac{\delta(1-\alpha_F)}{\alpha_R}$; note $\theta^* < 1$. Thus, $t(x) < 0$ for $x < \theta^*$ and $t(x) > 0$ for $x > \theta^*$. Also, $p(\cdot)$ is convex function such that $p(0) > 0$ and $p(1) > 0$ (recall $(\Delta_R)^a > 1$). Thus, there exists $\theta_1 = \frac{-\kappa - \sqrt{K_\delta}}{2A}, \theta_2 = \frac{-\kappa + \sqrt{K_\delta}}{2A}$ such that $p(\theta_1) = p(\theta_2) = 0$, provided $K_\delta \geq 0$. By convexity, $p(1) > 0$ and $p(0) > 0$, either both $\theta_1, \theta_2$ are above $1$, or below $0$, or are in $(0,1)$. With the above notations, 
    $$
        \widetilde{\theta} = \min\{ \max\{\theta^*, \theta_2\} + \epsilon , 1\}, \mbox{ with } \epsilon > \max \{0, \theta - \theta_2\}.
    $$

    If ${\widetilde \theta} = 1$, then clearly $\theta \leq {\widetilde \theta} = 1$. If $g(\theta) = 0$, then, $p(\theta) = 0$. Thus, either $\theta = \theta_1$ or $\theta = \theta_2$. Therefore, ${\widetilde \theta} = \max\{\theta_2, \theta^*\} + \epsilon > \theta_2 \geq \theta$. Else if $g(\theta) < 0$, then we will prove the claim for three cases separately. 

    Case 1: If $\theta > \theta^*$. Then $t(\theta) > 0$. Also, $g(\theta) < 0$, therefore, $p(\theta) < 0$. Thus, $\theta_1, \theta_2 \in (0,1)$ and $\theta \in (\theta_1, \theta_2)$.     By definition of ${\widetilde \theta}$, in this case,
    ${\widetilde \theta} = \theta_2 + \epsilon > \theta_2 > \theta$. 

    Case 2: If $\theta < \theta^*$. Then $t(\theta) < 0$, and $g(\theta) < 0$. Thus, $p(\theta) > 0$, which implies  either $\theta < \theta_1$ or $\theta > \theta_2$ (by convexity, $p(x) < 0$ for $x \in (\theta_1, \theta_2)$). Again by definition of ${\widetilde \theta}$, in this case we have:

    (i) ${\widetilde \theta} = \max\{\theta^*, \theta_2\} + \epsilon > \max\{\theta^*, \theta_2\} > \theta$ if $\theta < \theta_1$, or 

    (ii) ${\widetilde \theta} = \max\{\theta^*, \theta_2\} + \epsilon > \theta^* + \theta - \theta_2 > \theta$ if $\theta > \theta_2$.

    Case 3: If $\theta = \theta^*$. Then $t(\theta) = 0$ and $g(\theta) < 0$. Thus, $p(\theta) < 0$, and the claim follows as in case 1. 
\end{proof}

\begin{lemma}\label{lemma_feasibility_w}
    Under the hypothesis of Theorem \ref{thrm_response1} and for notations as in Algorithm \ref{alg_AI} when $K_\delta \geq 0$, the choice of $w$ is feasible.
\end{lemma}
\begin{proof}
    We are given $w$ such that:
\begin{align*}
cw\alpha_R &\in \bigg(\frac{1}{1-\mua} \max\left\{1, \frac{1}{(\Delta_R)^a \widetilde{\theta}}\right\},\\
&\hspace{2cm}\min\left\{  \frac{1}{\delta_a}, \fone \right\}  \bigg).
\end{align*}
We will show that the above interval is not empty. 

\noindent (i) If $\frac{1}{\delta_a} \leq \fone$, then:

$\bullet$ $\frac{1}{1-\mua} < \frac{1}{\delta_a} = \frac{1}{\delta(1-\mua)}$ since $\delta < 1$.

$\bullet$ under hypothesis of Theorem \ref{thrm_response1} and by Lemma \ref{lemma_theta_tilde}, $\widetilde{\theta} (\Delta_R)^a \geq \theta (\Delta_R)^a > \delta$. Thus, $\frac{1}{\widetilde{\theta} (\Delta_R)^a (1-\mua)} < \frac{1}{\delta_a}$.

\noindent (ii) If $\frac{1}{\delta_a} > \fone$, then:

$\bullet$ recall that $\alpha_R < \delta$ and $\eta^*_{\widetilde{\theta}} < 1-\mua$, therefore, 
\begin{align*}
    % 0 &<  \eta^*_{\widetilde{\theta}} (\delta-\alpha_R)\\
    % \implies
    0 = \delta(1-\mua) - \delta_a &< \eta^*_{\widetilde{\theta}} (\delta-\alpha_R)\\
    \implies \delta(1-\mua - \eta^*_{\widetilde{\theta}} ) &< \delta_a - \eta^*_{\widetilde{\theta}} \alpha_R\\
    \implies \frac{1}{1-\mua} &< \fone.
\end{align*}

$\bullet$ lastly, $\frac{1}{(\Delta_R)^a \widetilde{\theta} (1-\mua)} < \fone$ if:
\begin{align*}
   \frac{1}{(\Delta_R)^a \widetilde{\theta}} &< \frac{\delta_a - \eta^*_{\widetilde{\theta}} \alpha_R}{(1-\mua-\eta^*_{\widetilde{\theta}} )\delta}, \mbox{ i.e., if}\\
    % \widetilde{\theta} &> \frac{1}{(\Delta_R)^a} \left( \frac{\delta_a - \eta^*_{\widetilde{\theta}}  \delta}{\delta_a - \eta^*_{\widetilde{\theta}}  \alpha_R} \right), \mbox{ i.e., if}\\
   \widetilde{\theta} (\Delta_R)^a (\delta_a - \eta^*_{\widetilde{\theta}}  \alpha_R) &> \delta_a - \eta^*_{\widetilde{\theta}}  \delta, \mbox{ i.e., if}\\
    \eta^*_{\widetilde{\theta}} (\delta - \alpha_R   (\Delta_R)^a \widetilde{\theta} ) &> \delta_a (1-  (\Delta_R)^a\widetilde{\theta}), \mbox{ i.e., if}\\
    \left(\frac{1-\widetilde{\theta}}{1-\alpha_F}\right) (\delta - \alpha_R   (\Delta_R)^a \widetilde{\theta} )  &> \delta (1-  (\Delta_R)^a\widetilde{\theta}), \mbox{ i.e., if}\\
     % (1-\widetilde{\theta}) (\delta - \alpha_R   (\Delta_R)^a \widetilde{\theta} )   &>\delta (1-  (\Delta_R)^a\widetilde{\theta}) (1-\alpha_F), \mbox{ i.e., if}\\
     p(\widetilde{\theta}) &>0,
 \end{align*}
 for $p(\cdot)$ defined in the proof of Lemma \ref{lemma_theta_tilde}. %In the above, inequality $a$ holds true as (recall $\delta > \alpha_R$):
% \begin{align*}
%     \delta_a - \eta^*_{\widetilde{\theta}} \alpha_R &= (1-\mua) \left(\delta - \frac{(1-\widetilde{\theta}) \alpha_R}{1-\alpha_F} \right) \\ 
%     &>(1-\mua)\alpha_R \left(\frac{\widetilde{\theta} - \alpha_F}{1-\alpha_F}\right) > 0
% \end{align*}
% Now, we need to show that $p(\widetilde{\theta}) > 0$.  
Recall from the proof of Lemma \ref{lemma_theta_tilde} that the two zeroes, $\theta_1, \theta_2$, of convex function $p(\cdot)$ are either above $1$, or below $0$, or are in $(0,1)$. Further, $p(0) > 0$ and $p(1) > 0$. In the first two cases, $p(x) > 0$ for all $x \in [0,1]$; thus $p(\widetilde{\theta}) > 0$. In the last case, by definition of $\widetilde{\theta}$, $\widetilde{\theta} > \theta_2$, thus $p(\widetilde{\theta}) > 0$.
\end{proof}

\end{document}